\documentclass[12pt,a4paper,reqno]{amsart}

\usepackage[a4paper, margin=24mm, bottom=24mm, top=27mm]{geometry}
\usepackage{amsmath,amssymb,amsthm,amscd}
\usepackage{mathtools}

\usepackage[T1]{fontenc}
\usepackage{newtxtext,newtxmath}
\usepackage[colorlinks=true,linkcolor=cyan,citecolor=magenta,urlcolor=blue]{hyperref}
\usepackage{enumerate}
\usepackage[all,ps]{xypic}
\usepackage{tikz-cd}

\newtheorem{theorem}{Theorem}[section]
\newtheorem{proposition}[theorem]{Proposition}
\newtheorem{lemma}[theorem]{Lemma}
\newtheorem{corollary}[theorem]{Corollary}
\theoremstyle{definition}
\newtheorem{definition}[theorem]{Definition}
\newtheorem{remark}[theorem]{Remark}
\newtheorem{example}[theorem]{Example}

\newtheorem{construction}[theorem]{Construction}

\newcommand{\OrdFor}{\mathbf{OrdFor}}
\newcommand{\Quiv}{\mathbf{Quiv}}

\newcommand{\sC}{s\mathcal{C}}
\newcommand{\For}{\mathbf{For}}

\newcommand{\arb}[1]{\mathrm{arb}{#1}}
\newcommand{\ArbMod}[1]{\mathrm{arb}\mathbf{Mod}_{#1}}
\newcommand{\Mod}[1]{\mathbf{Mod}\text{-}{#1}}
\newcommand{\sMod}[1]{\mathbf{sMod}_{#1}}

\newcommand{\Fib}{\mathrm{Fib}}
\newcommand{\Cof}{\mathrm{Cof}}
\newcommand{\Hom}{\mathrm{Hom}}
\newcommand{\Mor}{\mathrm{Mor}}
\newcommand{\Ho}{\mathrm{ho}}
\newcommand{\Ch}{\mathbf{Ch}}

\numberwithin{equation}{section}

\title{Arboreal Objects and Their Homotopy Theory}
\author{Atabey Kaygun}
\address{Istanbul Technical University, Istanbul}
\email{kaygun@itu.edu.tr}

\begin{document}
\maketitle

\begin{abstract}
  We construct a category $\OrdFor$ as an arboreal extension of
  $\Delta_{\mathrm{epi}}\subseteq\Delta$, whose morphisms are ordered forests composed by
  grafting. We define a full functor $\pi\colon \OrdFor\to\Delta_{\mathrm{epi}}^{op}$ extracting
  the semisimplicial shadow. For every complete category $\mathcal C$, this induces a fully
  faithful functor from semisimplicial objects in $\mathcal C$ to $\mathcal C$-valued presheaves on
  $\OrdFor$, with right adjoint given by right Kan extension. We show that if weak equivalences of
  arboreal objects are detected by this right adjoint, then their Gabriel--Zisman localization is
  equivalent to that of semisimplicial objects. For bicomplete cofibrantly generated model
  categories, under the usual acyclicity hypothesis for right-induced transfer, the corresponding
  model structure on arboreal objects is Quillen equivalent to the Reedy model structure on
  semisimplicial objects.
\end{abstract}

\section*{Introduction}

In this paper we construct a category $\OrdFor$ of ordered forests, defined in purely
order-theoretic terms, as an arboreal extension of the subcategory
$\Delta_{\mathrm{epi}}\subseteq\Delta$ of surjections of the simplex category. Its morphisms are
finite posets equipped with compatible total orders and satisfying an interval condition on
principal lower sets. Every ordered forest decomposes canonically into ordered trees, and
boundary-labelled forests compose by grafting along maximal and minimal strata. Thus $\OrdFor$
provides a tree-like category built from classical order-theoretic data.

This places $\OrdFor$ near the planar side of dendroidal combinatorics, but not inside the usual
operadic framework. In the dendroidal category $\Omega$, morphisms are induced by maps of free
operads on rooted trees \cite{MoerdijkWeiss2007,HeutsMoerdijk2022}, so composition is governed by
operadic substitution at vertices and therefore has an essentially many-to-one character. In
$\OrdFor$, by contrast, morphisms are boundary-labelled ordered forests, and composition is given
by external grafting along maximal and minimal strata via a restricted cospan calculus. Internal
vertices are not altered. The resulting combinatorics is therefore closer to a many-to-many gluing
formalism, analogous to that of PROPs, than to operadic substitution. Our aim is not to reproduce
$\Omega$, but to isolate an order-theoretic category of rooted forests with boundary grafting and a
canonical linear, semisimplicial shadow. In this sense, the relevant formal background is closer to
the compositional theory of cospans \cite{Lack2004,RosebrughSabadiniWalters2005} than to operadic
composition, while the use of compatible total orders may be viewed as a combinatorial
rigidification in the spirit of Berger \cite{Berger2022}.

The link with linear combinatorics is given by a full functor
$\pi\colon \OrdFor\longrightarrow \Delta_{\mathrm{epi}}^{op}$ which sends an ordered forest to the
surjection determined by its minimal and maximal strata. For any complete category $\mathcal C$,
precomposition with $\pi^{op}$ yields a fully faithful functor
$\pi^\ast\colon \mathcal C^{\Delta_{\mathrm{epi}}}\longrightarrow \mathcal C^{\OrdFor^{op}}$ with
right adjoint $\pi_\ast$ given by right Kan extension. The homotopical point of the paper is that
$\pi_\ast$ does more than compare arboreal and semisimplicial objects: it detects the weak
equivalences relevant for the arboreal homotopy theory.

This has two consequences. First, at the level of Gabriel--Zisman localization, whenever weak
equivalences of arboreal objects are defined by detection through $\pi_\ast$, the resulting
homotopy category is equivalent to that of semisimplicial objects. Second, if $\mathcal C$ is a
bicomplete cofibrantly generated model category and the standard acyclicity condition for
right-induced transfer holds, then $\arb{\mathcal C}=\mathcal C^{\OrdFor^{op}}$ admits the
right-induced model structure along $\pi_\ast$, and the adjunction
\[
  \pi^\ast\colon \mathcal C^{\Delta_{\mathrm{epi}}}
  \rightleftarrows
  \mathcal C^{\OrdFor^{op}}:\pi_\ast
\]
is a Quillen equivalence. Thus the additional arboreal combinatorics carried by $\OrdFor$ does not
change the resulting homotopy theory once weak equivalences are measured through the semisimplicial
shadow.

\subsection*{What is known?}\label{subsec:what-is-known}

The starting point is classical simplicial theory. The simplicial category $\Delta$ is the basic
indexing category for simplicial constructions, and its subcategory
$\Delta_{\mathrm{epi}}\subseteq\Delta$ of surjections is equally standard. Passing to opposites
identifies $\Delta_{\mathrm{epi}}^{op}$ with $\Delta_{\mathrm{inj}}$, hence with the indexing
category for semisimplicial objects \cite{MacLane1998}.  Segal's category $\Gamma$ provides another
classical linear indexing category, adapted to commutative homotopy-coherent structures rather than
simplicial ones \cite{Segal1974}. In all these cases, the basic shapes are linear.

For operadic and higher-multiplicative phenomena, linear shapes are no longer adequate. Beginning
with Boardman--Vogt and continuing throughout modern operad theory, rooted trees replace simplices
as the relevant combinatorial shapes
\cite{BoardmanVogt1973,BergerMoerdijk2007,LodayVallette2012}. The standard categorical realization
of this principle is the dendroidal category $\Omega$ of Moerdijk and Weiss
\cite{MoerdijkWeiss2007}, whose objects are rooted trees and whose morphisms are maps between the
free operads generated by those trees. This is the main point of comparison with the present
work. Our category $\OrdFor$ likewise has arboreal flavour, but its morphisms are not operadic:
they are ordered forests with labelled boundary, and composition is defined by grafting. Thus
$\OrdFor$ is not a reformulation of $\Omega$, but a different arboreal indexing category with a
distinct morphism calculus.

A more explicitly order-theoretic approach to arboreal combinatorics appears in Weiss's broad
posets \cite{Weiss2012BroadPosets}, which show that tree-based indexing need not take operads as
primitive data. The present construction belongs to this direction, but with a different purpose:
rather than recasting operadic substitution in posetal language, it isolates a category of ordered
forests equipped with a canonical functor $\pi\colon \OrdFor\to\Delta_{\mathrm{epi}}^{op}$ linking
arboreal and linear indexing categories directly.

Berger's moment categories place $\Gamma$, $\Delta$, and dendroidal constructions in a common
formalism and emphasize comparison results among Segal-type presheaf theories
\cite{Berger2022}. Hackney studies graph-based indexing categories associated with generalized
operads and the corresponding presheaf theories \cite{Hackney2024}. Haderi and Ünlü work with
presheaves on levelled trees in a framework adapted to non-symmetric operads and simplicial-list
structures \cite{HaderiUnlu2024}. These approaches are close in spirit to the present one in that
they seek alternative indexing categories between the linear and operadic settings. What
distinguishes the present paper is the particular combination of features: ordered forests as
morphisms, grafting as composition, and a full functor
$\pi\colon \OrdFor\to\Delta_{\mathrm{epi}}^{op}$ through which the linear theory appears as a
shadow of the arboreal one.

The category $\OrdFor$ should also be compared with the binary magmatic formalism of
\cite{AskarogullariKaygun2024Leib}, where grafting operations of binary trees are encoded
algebraically by generators and relations. Here the same arboreal intuition is organized
categorically and extended from binary trees to arbitrary ordered forests. The homotopical
comparison proved here is also analogous to the phenomenon observed in \cite{KayaKaygun2024}: an
indexing category may carry additional combinatorial structure without changing the resulting
homotopy theory, provided that weak equivalences are detected by a suitable shadow functor. The
novelty of the present paper is to establish this for an arboreal extension of
$\Delta_{\mathrm{epi}}$: we construct the category $\OrdFor$ and the full functor
$\pi\colon \OrdFor\to\Delta_{\mathrm{epi}}^{op}$, and we show that when weak equivalences of
arboreal objects are detected by $\pi_\ast$, their homotopy theory agrees with that of
semisimplicial objects.

\subsection*{Plan of the article}

Section~\ref{sec:ordered-quivers} develops the basic combinatorics. We define ordered quivers,
ordered forests, and ordered trees in purely order-theoretic terms, describe their boundary strata,
and prove that every ordered forest decomposes canonically into an ordered disjoint union of
ordered trees.

Section~\ref{sec:OrdFor} constructs the category $\OrdFor$ of boundary-labelled ordered forests,
with composition given by grafting along maximal and minimal boundary strata.

Section~\ref{sec:height-one} identifies the linear shadow of the construction. We show that
height-one forests recover surjections on each hom-set and thereby obtain the full functor
$\pi\colon \OrdFor\to\Delta_{\mathrm{epi}}^{op}$.

Section~\ref{sec:arboreal-homotopy} studies arboreal objects
$\arb{\mathcal C}=\mathcal C^{\OrdFor^{op}}$.  We prove that localization at weak equivalences
detected by the right Kan extension $\pi_\ast$ agrees with the homotopy theory of semisimplicial
objects, and that any cofibrantly generated model structure on $\mathcal C^{\Delta_{\mathrm{epi}}}$
induces a right-transferred model structure on $\arb{\mathcal C}$ which is Quillen equivalent to
the semisimplicial one.

\section{Ordered Quivers, Forests, And Trees}\label{sec:ordered-quivers} 

\begin{definition}\label{def:ordered-quiver}
  An \emph{ordered quiver} is a triple $(P,\leq_P,\preceq_P)$ consisting of a finite set $P$, a
  partial order $\leq_P$, and a total order $\preceq_P$ such that $x\leq_P y$ implies
  $x\preceq_P y$. A morphism $f:(P,\leq_P,\preceq_P)\to(Q,\leq_Q,\preceq_Q)$ is a set map
  preserving both orders. We write $\Quiv$ for the resulting category.
\end{definition}

\subsection{Ordered forests and trees}

For a finite poset $(P,\leq_P)$ and $x\in P$, write $L_P(x):=\{y\in P\mid y\leq_P x\}$ and
$U_P(x):=\{y\in P\mid x\leq_P y\}$.

\begin{definition}\label{def:ordered-forest}
  Recall that a subset $I$ of a poset $(P,\leq)$ is called an interval if for every $a\leq b$ in
  $I$ we have $U_P(a)\cap L_P(b)\subseteq I$.  An \emph{ordered forest} is an ordered quiver
  $(P,\leq_P,\preceq_P)$ such that $L_P(x)$ is an interval in $(P,\preceq_P)$ for every $x\in
  P$. An \emph{ordered tree} is an ordered forest whose partial order has a unique maximal element.
\end{definition}

Let $\For\subseteq\Quiv$ denote the full subcategory spanned by the ordered forests.

\subsection{The decomposition theorem}\label{subsec:decomposition}

\begin{theorem}\label{thm:forest-decomposition}
  Let $(P,\leq_P,\preceq_P)$ be an ordered forest, and write $\max(P,\leq_P)=\{m_0,\dots,m_r\}$
  with $m_0\preceq_P\cdots\preceq_P m_r$.  Set $T_i:=L_P(m_i)$. Then:
  \begin{enumerate}[(i)]
  \item each $T_i$ is an ordered tree; \label{forest-i}
  \item the sets $T_0,\dots,T_r$ are pairwise disjoint; \label{forest-ii}
  \item $P=\bigcup_{i=0}^r T_i$; \label{forest-iii}
  \item each $T_i$ is an interval in $(P,\preceq_P)$; \label{forest-iv}
  \item $P$ is canonically isomorphic to the ordered disjoint union $T_0\oplus\cdots\oplus
    T_r$. \label{forest-v}
  \end{enumerate}
  In particular, every ordered forest decomposes canonically into an ordered disjoint union of
  ordered trees.
\end{theorem}

\begin{proof}
  We first show that every $x\in P$ lies below a unique maximal element. Since $P$ is finite, the
  upper set $U_P(x)$ contains at least one maximal element, so there exists $m\in\max(P,\leq_P)$
  with $x\leq_P m$. For uniqueness, suppose $x\leq_P m$ and $x\leq_P m'$ with $m,m'$ maximal. After
  relabeling, assume $m\preceq_P m'$. Since $x,m'\in L_P(m')$ and $L_P(m')$ is an interval in
  $(P,\preceq_P)$, it follows that $m\in L_P(m')$. Hence $m\leq_P m'$, and therefore $m=m'$ by
  maximality.

  For \eqref{forest-i}, fix $i$. Since $T_i=L_P(m_i)$ and lower sets in an ordered forest are
  intervals, $T_i$ is again an ordered forest with the induced orders. Moreover, $m_i$ is maximal
  in $T_i$, and it is the unique maximal element: if $x\in T_i$ is maximal in $T_i$, then
  $x\leq_P m_i$, hence $x=m_i$. Thus $T_i$ is an ordered tree.

  Assertion \eqref{forest-ii} follows from uniqueness. Indeed, if $x\in T_i\cap T_j$, then
  $x\leq_P m_i$ and $x\leq_P m_j$, so the maximal element above $x$ is not unique unless $m_i=m_j$.
  Hence $i=j$.

  Assertion \eqref{forest-iii} follows from existence: every $x\in P$ lies below some maximal
  element $m_i$, so $x\in T_i$.

  Assertion \eqref{forest-iv} is immediate from the definition $T_i=L_P(m_i)$ and the defining
  property of an ordered forest.

  For \eqref{forest-v}, the sets $T_i$ form a partition of $P$ by \eqref{forest-ii} and
  \eqref{forest-iii}, and each $T_i$ is an interval in $(P,\preceq_P)$ by \eqref{forest-iv}. Since
  $m_0\preceq_P\cdots\preceq_P m_r$ and $m_i\in T_i$, these intervals occur in this order. Also,
  elements of distinct $T_i$ and $T_j$ are incomparable for $\leq_P$, since otherwise a point would
  lie below two distinct maximal elements, contradicting uniqueness. Thus the partial order on $P$
  is the disjoint union of the induced partial orders on the $T_i$, and the total order is their
  ordinal sum. Hence the identity map induces an isomorphism $P\cong T_0\oplus\cdots\oplus T_r$
  which proves \eqref{forest-v}.
\end{proof}
\section{The Cospan Category $\OrdFor$}\label{sec:OrdFor}

We define the category $\OrdFor$ of ordered forests with labelled boundary. Its composition is
based on a restricted cospan composition, refined by a reduction step to ensure finite
hom-sets. Compare \cite{Lack2004,RosebrughSabadiniWalters2005} for the general cospan formalism.

\begin{definition}\label{def:height}
  The height of an ordered forest $(F,\leq_F,\preceq_F)$ is the maximum length of a chain
  $x_0<_F \cdots<_F x_n$ in the underlying poset $(F,\leq_F)$.
\end{definition}

An ordered tree is of \emph{height-one} if every element is either minimal or
maximal. Equivalently, it has a unique maximal element and every other element is minimal. An
ordered forest is called \emph{height-one} if each of its ordered tree components is height-one.

An ordered forest is of \emph{height-zero} if it is a finite totally ordered set endowed with the
trivial partial order $x\leq_F y$ iff $x=y$. For each $n\geq 0$, let $[n]:=\{0,1,\dots,n\}$ with
its natural total order and the trivial partial order. We regard $[n]$ as a height-zero forest.

\begin{definition}\label{def:boundary-labelled-forest}
  Let $m,n\geq 0$. A \emph{boundary-labelled ordered forest from $[m]$ to $[n]$} is a cospan
  \[
    [m]\xrightarrow{\ \iota^{+}\ } F \xleftarrow{\ \iota^{-}\ } [n]
  \]
  in $\Quiv$ such that $F$ is an ordered forest, $\iota^{+}$ identifies $[m]$ with
  $\max(F,\leq_F)$, and $\iota^{-}$ identifies $[n]$ with $\min(F,\leq_F)$, both with the total
  orders induced by $\preceq_F$.
\end{definition}

\begin{definition}\label{def:isomorphism-cospans}
  Two boundary-labelled ordered forests
  $[m]\xrightarrow{\ \iota^{+}\ } F \xleftarrow{\ \iota^{-}\ } [n]$ and
  $[m]\xrightarrow{\ \jmath^{+}\ } F' \xleftarrow{\ \jmath^{-}\ } [n]$ are \emph{isomorphic} if
  there is an isomorphism $\alpha:F\to F'$ in $\Quiv$ such that the following diagram commutes
 \[\xymatrix{
      & F \ar@{..>}[dd]_{\alpha} \\
      [m] \ar[ur]^{\iota^+} \ar[dr]_{\jmath^+} & & \ar[dl]^{\jmath^-} \ar[ul]_{\iota^-} [n]\\
      & F'
    }\]
  In other words, $\alpha\iota^{+}=\jmath^{+}$ and $\alpha\iota^{-}=\jmath^{-}$.
\end{definition}

\begin{remark}\label{rem:direction}
  Our convention is that the domain is the ordered set of maximal elements and the codomain is the
  ordered set of minimal elements. With this choice, the linear shadow constructed later is a
  functor of the form $\pi\colon \OrdFor\to\Delta_{\mathrm{epi}}^{op}$.
\end{remark}

To build the composition in $\OrdFor$, we first define an intermediate grafting operation for
boundary-labelled forests.

\begin{definition}\label{def:composition}
  Let $[\ell]\xrightarrow{\ \iota_G^{+}\ } G \xleftarrow{\ \iota_G^{-}\ } [m]$ and
  $[m]\xrightarrow{\ \iota_F^{+}\ } F \xleftarrow{\ \iota_F^{-}\ } [n]$ be composable
  boundary-labelled ordered forests. For $i\in[m]$, write $x_i:=\iota_F^{+}(i)\in\max(F)$,
  $y_i:=\iota_G^{-}(i)\in\min(G)$, and $F_i:=L_F(x_i)$. By Theorem~\ref{thm:forest-decomposition},
  $F\cong F_0\oplus\cdots\oplus F_m$, and each $F_i$ is an interval in $(F,\preceq_F)$.

  Define the raw graft $H:=F\cdot G$ by replacing each $y_i\in\min(G)$ by the tree
  $F_i$. Concretely, $H=(G\setminus\min(G))\sqcup F$. Let $j_F:F\hookrightarrow H$ be the
  inclusion, and define $j_G:G\to H$ by $j_G(y_i)=x_i$ and $j_G(z)=z$ for $z\in
  G\setminus\min(G)$. Define also $\pi:H\to G$ by $\pi(a)=y_i$ for $a\in F_i$ and $\pi(a)=a$ for
  $a\in G\setminus\min(G)$.

  The total order on $H$ is defined by $a\preceq_H b$ if either $\pi(a)\prec_G\pi(b)$, or
  $\pi(a)=\pi(b)=y_i$ for some $i$ and $a\preceq_F b$, or $\pi(a)=\pi(b)\in G\setminus\min(G)$ and
  $a=b$.

  The partial order on $H$ is defined by:
  \begin{enumerate}[(i)]
  \item if $a,b\in F_i$, then $a\leq_H b$ iff $a\leq_F b$;
  \item if $a,b\in G\setminus\min(G)$, then $a\leq_H b$ iff $a\leq_G b$;
  \item if $a\in F_i$ and $b\in G\setminus\min(G)$, then $a\leq_H b$ iff $y_i\leq_G b$;
  \item no other pairs are comparable.
  \end{enumerate}
  The induced maps $[\ell]\xrightarrow{j_G\iota_G^{+}}H$ and $[n]\xrightarrow{j_F\iota_F^{-}}H$ are the boundary labels of the raw composite.
\end{definition}

\begin{proposition}\label{prop:composition-ordered-forest}
  The raw composite $H=F\cdot G$ of Definition~\ref{def:composition} is an ordered forest. Its
  maximal elements are $j_G(\max(G))$, its minimal elements are $j_F(\min(F))$, and hence it
  defines a boundary-labelled ordered forest from $[\ell]$ to $[n]$.
\end{proposition}

\begin{proof}
  By construction, $\preceq_H$ is a total order. The relation $\leq_H$ is a partial order, as the
  only nontrivial cases are mixed ones, and these reduce immediately to transitivity in $G$. Also
  $a\leq_H b$ implies $a\preceq_H b$, since this is clear inside each $F_i$ and inside
  $G\setminus\min(G)$, while in the mixed case $a\in F_i$, $b\in G\setminus\min(G)$ one has
  $y_i\leq_G b$, hence $y_i\preceq_G b$. Now let $u\in H$. If $u\in F_i$, then $L_H(u)=L_F(u)$, so
  $L_H(u)$ is an interval. If $u\in G\setminus\min(G)$, then $L_H(u)=\pi^{-1}(L_G(u))$, which is
  again an interval, since $L_G(u)$ is an interval in $G$ and each fibre $\pi^{-1}(y_i)=F_i$ is
  inserted at the position of $y_i$. Thus $H$ is an ordered forest. Finally, the maximal elements
  of $H$ are exactly the maximal elements of $G$, hence $j_G(\max(G))$, and the minimal elements
  are exactly the minimal elements of the inserted trees $F_i$, hence $j_F(\min(F))$.
\end{proof}

\begin{definition}\label{def:boundary-compatible-cocone}
  Let $F \xleftarrow{\ \iota_F^{+}\ } [m] \xrightarrow{\ \iota_G^{-}\ } G$ be the span underlying a
  composable pair, and write $x_i:=\iota_F^{+}(i)$, $y_i:=\iota_G^{-}(i)$, and $F_i:=L_F(x_i)$. A
  \emph{boundary-compatible cocone} from this span to an ordered forest $Q$ is a pair of morphisms
  $f:F\to Q$ and $g:G\to Q$ in $\For$ such that $f\iota_F^{+}=g\iota_G^{-}$ and, for every
  $i\in[m]$, every $z\in G$ with $z\prec_G y_i$, and every $a\in F_i$, one has
  $g(z)\preceq_Q f(a)$.
\end{definition}

\begin{proposition}\label{prop:composition-universal}
  Let $[\ell]\xrightarrow{\ \iota_G^{+}\ } G \xleftarrow{\ \iota_G^{-}\ } [m]$ and
  $[m]\xrightarrow{\ \iota_F^{+}\ } F \xleftarrow{\ \iota_F^{-}\ } [n]$ be composable, and let
  $H=F\cdot G$. Then
  \[\begin{CD}
      [m] @>{\iota_G^{-}}>> G\\
      @V{\iota_F^{+}}VV @VV{j_G}V\\
      F @>{j_F}>> H
    \end{CD}\]\\
  is universal among boundary-compatible cocones. Equivalently, every boundary-compatible cocone
  $f:F\to Q$, $g:G\to Q$ factors uniquely through a morphism $u:H\to Q$ in $\For$ with $uj_F=f$ and
  $uj_G=g$.
\end{proposition}

\begin{proof}
  Define $u:H\to Q$ by $u(a)=f(a)$ for $a\in F$ and $u(a)=g(a)$ for $a\in G\setminus\min(G)$. This
  is well defined because the only identifications are $j_G(y_i)=x_i$, and there one has
  $g(y_i)=f(x_i)$. It remains to check that $u$ is a morphism in $\For$. The only mixed
  partial-order case is $a\in F_i$, $b\in G\setminus\min(G)$ with $a\leq_H b$. Then $y_i\leq_G b$,
  so $f(a)\leq_Q f(x_i)=g(y_i)\leq_Q g(b)$. For the total order, the only mixed cases are again
  $a\in F_i$, $b\in G\setminus\min(G)$. If $a\preceq_H b$, then
  $f(a)\preceq_Q f(x_i)=g(y_i)\preceq_Q g(b)$. If $b\preceq_H a$, then necessarily $b\prec_G y_i$,
  and boundary-compatibility gives $g(b)\preceq_Q f(a)$. Thus $u$ preserves both orders, and
  uniqueness is immediate.
\end{proof}

While the raw grafting operation $F \cdot G$ is perfectly well-behaved, defining the morphisms of
$\OrdFor$ to be all possible boundary-labelled forests would result in infinite hom-sets due to the
possibility of unbounded non-branching chains. To ensure the category is structurally bounded and
suitable for homotopical transfer via right Kan extension, we rigorously restrict our morphisms to
reduced forests.

\begin{definition}\label{def:unary-vertex}
  Let $(F,\leq_F,\preceq_F)$ be an ordered forest. An element $v \in F$ is called an \emph{internal
    vertex} if $v \notin \min(F,\leq_F) \cup \max(F,\leq_F)$. An internal vertex $v$ is
  \emph{unary} if it has exactly one immediate predecessor and exactly one immediate successor in
  the partial order $\leq_F$. An ordered forest is \emph{reduced} if it contains no unary internal
  vertices.
\end{definition}

\begin{definition}\label{def:reduction}
  Let $H$ be an ordered forest. The \emph{reduction} of $H$, denoted $\mathrm{Red}(H)$, is the
  ordered forest obtained by iteratively contracting all unary internal vertices. Explicitly, if
  $v \in H$ is a unary internal vertex with immediate predecessor $u$ and immediate successor $w$,
  we remove $v$ from the underlying set and replace the covering relations $u <_H v$ and $v <_H w$
  with a single covering relation $u <_H w$, while inheriting the total order $\preceq_H$ on the
  remaining elements.
\end{definition}

The reduction operation iteratively collapses unary internal vertices, thus ensure that every
internal vertex in a reduced forest has a branching factor of at least two. Since the number of
leaves and roots is strictly determined by the domain and codomain boundaries, the branching
condition guarantees an upper bound on the total number of internal vertices. Consequently, the
number of isomorphism classes of reduced forests between any two boundaries is finite, i.e. the
hom-sets in $\OrdFor$ finite.

\begin{definition}\label{def:OrdFor-morphisms}
  The objects of $\OrdFor$ are the finite ordinals $[n]$, viewed as height-zero ordered forests. A
  morphism $[m]\to[n]$ in $\OrdFor$ is an isomorphism class of \emph{reduced} boundary-labelled
  ordered forests $[m]\xrightarrow{\ \iota^{+}\ } F \xleftarrow{\ \iota^{-}\ } [n]$.
\end{definition}

\begin{definition}\label{def:composition-reduced}
  Let $[\ell]\xrightarrow{\ \iota_G^{+}\ } G \xleftarrow{\ \iota_G^{-}\ } [m]$ and
  $[m]\xrightarrow{\ \iota_F^{+}\ } F \xleftarrow{\ \iota_F^{-}\ } [n]$ be composable
  morphisms. Their categorical composite is defined as $F \circ G := \mathrm{Red}(F \cdot G)$,
  where $F \cdot G$ is the raw grafting construction.
\end{definition}

\begin{definition}\label{def:identities}
  For each $n\geq 0$, the identity morphism of $[n]$ is represented by the trivially reduced
  height-zero forest $[n]\xrightarrow{\ id\ } [n] \xleftarrow{\ id\ } [n]$.
\end{definition}

\begin{theorem}\label{thm:OrdFor-category}
  The objects and morphisms defined above, equipped with the composition $F \circ G$, form a well-defined category $\OrdFor$. Furthermore, for any objects $[m]$ and $[n]$, the hom-set $\Hom_{\OrdFor}([m], [n])$ is finite.
\end{theorem}

\begin{proof}
  The identity cospan acts trivially under the grafting rule and is already reduced, naturally
  satisfying the unit axioms. To establish associativity, let $[r]\to K\leftarrow[\ell]$,
  $[\ell]\to G\leftarrow[m]$, and $[m]\to F\leftarrow[n]$ be composable morphisms. The raw grafting
  operation on ordered forests is strictly associative up to canonical isomorphism. The contraction
  process is a locally confluent and terminating rewriting system. By the Diamond Lemma, the
  reduction yields a unique reduced normal form independent of the order of vertex contractions are
  executed. Consequently, the reduction after each grafting step yields a canonically isomorphic
  result ensuring the associativity $(F \circ G) \circ K \cong F \circ (G \circ K)$. For the
  finiteness assertion, observe that every internal vertex in a reduced boundary-labelled forest
  must have a branching factor of at least two in either the upward or downward direction, and
  therefore, the total number of internal vertices is globally bounded by a function depending only
  on $m$ and $n$. Thus the hom-set between any two objects is strictly finite.
\end{proof}

\section{Height-one Forests}\label{sec:height-one}

We now identify the linear shadow of $\OrdFor$ by showing that height-one forests encode exactly
the surjections in $\Delta$.  Recall that the category $\Delta_{\mathrm{epi}}$ has as objects the
finite ordinals $[n]=\{0,\dots,n\}$ and as morphisms the order-preserving surjections.

\begin{construction}\label{cons:surjection-to-height-one}
  Let $\sigma\colon [n]\twoheadrightarrow[m]$ be an order-preserving surjection. For each
  $0\leq i\leq m$, let $B_i=\sigma^{-1}(i)$. Since $\sigma$ is order-preserving and surjective, each
  $B_i$ is a nonempty interval in $[n]$, and $[n]$ is the ordered disjoint union
  $B_0\oplus\cdots\oplus B_m$.

  For each $i\in[m]$, let $T_i$ be the height-one ordered tree whose unique maximal element is
  denoted $r_i$ and whose minimal elements are indexed, in their induced order, by the elements of
  $B_i$. Set
  \[
    F_\sigma:=T_0\oplus\cdots\oplus T_m.
  \]
  The ordered set of maximal elements of $F_\sigma$ is canonically identified with $[m]$ via
  $i\mapsto r_i$, and the ordered set of minimal elements is canonically identified with
  $[n]$. Hence $F_\sigma$ defines a morphism $[m]\to[n]$ in $\OrdFor$.
\end{construction}

\begin{construction}\label{cons:forest-to-shadow}
  Let $[m]\xrightarrow{\ \iota^{+}\ } F \xleftarrow{\ \iota^{-}\ } [n]$ be a boundary-labelled
  ordered forest. Write
  \[
    \iota^{+}(0)=r_0\prec_F\cdots\prec_F r_m
  \]
  for the maximal elements of $F$, and
  \[
    \iota^{-}(0)\prec_F\cdots\prec_F \iota^{-}(n)
  \]
  for the minimal elements of $F$. By Theorem~\ref{thm:forest-decomposition} we have a canonical
  decomposition
  \[
    F\cong L_F(r_0)\oplus\cdots\oplus L_F(r_m).
  \]
  Define a map $\sigma_F\colon [n]\to[m]$ by the rule
  \[
    \sigma_F(j)=i
    \qquad\Longleftrightarrow\qquad
    \iota^{-}(j)\in L_F(r_i),
  \]
  equivalently,
  \[
    \sigma_F(j)=i
    \qquad\Longleftrightarrow\qquad
    \iota^{-}(j)\leq_F r_i.
  \]

  Equivalently, one may form the full subposet
  \[
    H_F:=\min(F,\leq_F)\cup\max(F,\leq_F)\subseteq F
  \]
  with the induced partial order and total order. Then $H_F$ is a height-one ordered forest with
  the same boundary labels as $F$, and $\sigma_F$ is precisely the surjection associated to $H_F$.
\end{construction}

\begin{lemma}\label{lem:forest-shadow-surjection}
  For every boundary-labelled ordered forest $F$, the map $\sigma_F$ of
  Construction~\ref{cons:forest-to-shadow} is an order-preserving surjection.
\end{lemma}

\begin{proof}
  Every maximal element $r_i$ lies in a unique ordered-tree component of $F$, and every finite
  ordered tree has at least one minimal element. Hence for each $i$ there exists some $j\in[n]$
  with $\iota^{-}(j)\leq_F r_i$, so $\sigma_F$ is surjective.

  By Theorem~\ref{thm:forest-decomposition}, the lower set $L_F(r_i)$ is an interval in the ambient
  total order on $F$. Therefore, if
  \[
    \iota^{-}(a)\prec_F \iota^{-}(b)\prec_F \iota^{-}(c)
  \]
  and $\sigma_F(a)=\sigma_F(c)=i$, then both $\iota^{-}(a)$ and $\iota^{-}(c)$ lie in the interval
  $L_F(r_i)$, hence so does $\iota^{-}(b)$. Thus $\sigma_F(b)=i$. It follows that each fibre of
  $\sigma_F$ is an interval in $[n]$, and therefore $\sigma_F$ is order-preserving.
\end{proof}

\begin{proposition}\label{prop:height-one-bijection}
  For each pair of objects $[m]$ and $[n]$, the assignments $\sigma\mapsto F_\sigma$ and
  $F\mapsto \sigma_F$ induce mutually inverse correspondences between order-preserving surjections
  $[n]\twoheadrightarrow[m]$ and isomorphism classes of boundary-labelled height-one ordered
  forests $[m]\to F\leftarrow[n]$.
\end{proposition}

\begin{proof}
  Let $\sigma\colon [n]\twoheadrightarrow[m]$ be an order-preserving surjection. By construction,
  the forest $F_\sigma$ is the ordered disjoint union of the height-one trees indexed by the fibres
  $B_i=\sigma^{-1}(i)$. Applying Construction~\ref{cons:forest-to-shadow} recovers exactly the
  original fibre decomposition, so $\sigma_{F_\sigma}=\sigma$.

  Conversely, let $[m]\xrightarrow{\ \iota^{+}\ } F \xleftarrow{\ \iota^{-}\ } [n]$ be a
  boundary-labelled height-one ordered forest. The map $\sigma_F$ records, for each minimal element
  of $F$, the unique maximal element lying above it. Reconstructing $F_{\sigma_F}$ from these fibres
  produces the same ordered disjoint union of height-one trees. Hence $F_{\sigma_F}$ is canonically
  isomorphic to $F$ as a boundary-labelled ordered forest.
\end{proof}

\begin{theorem}\label{thm:height-one-deltaepi}
  There is a functor $\pi\colon \OrdFor\longrightarrow \Delta_{\mathrm{epi}}^{op}$ which is the
  identity on objects and sends a morphism represented by a boundary-labelled ordered forest
  $[m]\xrightarrow{\ \iota^{+}\ } F \xleftarrow{\ \iota^{-}\ } [n]$ to the order-preserving
  surjection $\sigma_F\colon [n]\twoheadrightarrow[m]$, regarded as a morphism $[m]\to[n]$ in
  $\Delta_{\mathrm{epi}}^{op}$. This functor is full. Moreover, for each pair of objects $[m]$ and
  $[n]$, it induces on height-one representatives the bijection of
  Proposition~\ref{prop:height-one-bijection}.
\end{theorem}

\begin{proof}
  If two boundary-labelled ordered forests are isomorphic, then the isomorphism preserves minimal
  elements, maximal elements, the partial order, the total order, and the boundary labellings.
  Therefore the associated surjections coincide. Hence $\sigma_F$ depends only on the isomorphism
  class of the morphism represented by $F$.

  The identity morphism of $[n]$ is represented by
  $[n]\xrightarrow{\ id\ } [n] \xleftarrow{\ id\ } [n]$. Its associated surjection is the identity
  of $[n]$. Thus $\pi$ preserves identities.

  Let $[\ell]\xrightarrow{}G\xleftarrow{}[m]$ and $[m]\xrightarrow{}F\xleftarrow{}[n]$ represent
  composable morphisms. Suppose $\sigma_F(j)=i$ and $\sigma_G(i)=k$. Then
  $\iota_F^-(j)\leq_F \iota_F^+(i)$ and $\iota_G^-(i)\leq_G \iota_G^+(k)$.  In the standard
  representative of the composite $F\circ G$, the elements $\iota_F^+(i)$ and $\iota_G^-(i)$ are
  identified. Hence the image of $\iota_F^-(j)$ in the composite lies below $\iota_G^+(k)$, so
  $\sigma_{F\circ G}(j)=k$. Therefore
  \[
    \sigma_{F\circ G}=\sigma_G\circ\sigma_F.
  \]
  Since composition in $\Delta_{\mathrm{epi}}^{op}$ is opposite to composition in
  $\Delta_{\mathrm{epi}}$, this is exactly the statement that
  \[
    \pi(F\circ G)=\pi(F)\circ\pi(G).
  \]
  Hence $\pi$ is a functor.

  To prove fullness, let $[m]\to[n]$ be any morphism in $\Delta_{\mathrm{epi}}^{op}$. Equivalently,
  let $\sigma\colon [n]\twoheadrightarrow[m]$ be the corresponding morphism in
  $\Delta_{\mathrm{epi}}$. Construction~\ref{cons:surjection-to-height-one} produces a height-one
  ordered forest $F_\sigma$ with $\pi(F_\sigma)=\sigma$, viewed in the opposite category. Hence
  every morphism of $\Delta_{\mathrm{epi}}^{op}$ lies in the image of $\pi$.  The final assertion
  is exactly Proposition~\ref{prop:height-one-bijection}.
\end{proof}

\begin{remark}\label{rem:simplicial-shadow}
  The functor $\pi$ is the linear shadow of the ordered-forest calculus. It forgets all internal
  branching data and remembers only, for each minimal boundary element, which maximal boundary
  element lies above it. For each morphism in $\Delta_{\mathrm{epi}}^{op}$, height-one forests
  provide canonical representatives in $\OrdFor$, but these representatives are not closed under
  composition. Thus $\Delta_{\mathrm{epi}}^{op}$ is recovered only as a quotient shadow of
  $\OrdFor$, not as a subcategory of it.
\end{remark}

\section{Homotopy Theory of Arboreal Objects}\label{sec:arboreal-homotopy}

Let $\pi\colon \OrdFor \longrightarrow \Delta_{\mathrm{epi}}^{op}$ be the full functor of
Theorem~\ref{thm:height-one-deltaepi}. For any category $\mathcal C$, write
\[
  \sC:=\mathcal C^{\Delta_{\mathrm{epi}}} \quad\text{ and }\quad \arb{\mathcal C}:=\mathcal
  C^{\OrdFor^{op}}
\]
for the categories of semisimplicial objects and arboreal objects in $\mathcal C$. We first compare
the two homotopy theories at the level of localization, and then pass to model structures.

\subsection{Equivalence of homotopy localizations}\label{subsec:localization-equivalence}

Assume first that $\mathcal C$ is complete. Precomposition with
$\pi^{op}\colon \OrdFor^{op}\longrightarrow \Delta_{\mathrm{epi}}$ defines a functor
\[
  \pi^\ast\colon \sC\longrightarrow \arb{\mathcal C},
  \qquad
  \pi^\ast(A)=A\circ \pi^{op},
\]
and since $\mathcal C$ is complete, $\pi^\ast$ admits a right adjoint
$\pi_\ast\colon \arb{\mathcal C}\longrightarrow \sC$ given by right Kan extension along $\pi^{op}$;
see~\cite[Chapter~X, \S 3]{MacLane1998}.

\begin{lemma}\label{lem:pi-star-fully-faithful}
  The functor $\pi^\ast\colon \sC\longrightarrow \arb{\mathcal C}$ is fully faithful.
\end{lemma}

\begin{proof}
  Since $\pi^{op}\colon \OrdFor^{op}\to \Delta_{\mathrm{epi}}$ is full and the identity on objects,
  precomposition with $\pi^{op}$ is fully faithful. Hence
  $\pi^\ast\colon \sC\longrightarrow \arb{\mathcal C}$ is fully faithful.
\end{proof}

Let $W_{\sC}$ be any class of morphisms in $\sC$. Define the corresponding class of morphisms in
$\arb{\mathcal C}$ by detection through the right Kan extension:
\[
  W_{\arb{\mathcal C}}
  :=
  \{\,f\in \Mor(\arb{\mathcal C})\mid \pi_\ast f\in W_{\sC}\,\}.
\]
Write
\[
  \Ho(\sC):=\sC[W_{\sC}^{-1}]
  \qquad\text{and}\qquad
  \Ho(\arb{\mathcal C})
  :=
  \arb{\mathcal C}[W_{\arb{\mathcal C}}^{-1}]
\]
for the corresponding Gabriel--Zisman localizations \cite[I.1]{GabrielZisman1967}.

\begin{theorem}\label{thm:localization-equivalence}
  Assume that $\mathcal C$ is complete. Then the adjunction
  \begin{equation}\label{eq:adjoint-equivalence}
    \pi^\ast\colon \sC \rightleftarrows \arb{\mathcal C}:\pi_\ast
  \end{equation}
  descends to an adjoint equivalence $\Ho(\sC)\simeq \Ho(\arb{\mathcal C})$.
\end{theorem}

\begin{proof}
  By Lemma~\ref{lem:pi-star-fully-faithful}, the unit $\eta_A\colon A\to \pi_\ast\pi^\ast A$ is an
  isomorphism for every $A\in \sC$. Hence if $u\in W_{\sC}$, then $\pi_\ast(\pi^\ast u)\cong u$, so
  $\pi^\ast u\in W_{\arb{\mathcal C}}$. By definition, $\pi_\ast$ sends $W_{\arb{\mathcal C}}$ into
  $W_{\sC}$. Therefore the adjunction $\pi^\ast\dashv \pi_\ast$ induces an adjunction on
  localizations
  \begin{equation}\label{eq:ho-equivalence}
    L(\pi^\ast)\colon \Ho(\sC)\rightleftarrows \Ho(\arb{\mathcal C}):L(\pi_\ast).
  \end{equation}
  The localized unit is the image of $\eta$, hence is an isomorphism. For $X\in \arb{\mathcal C}$,
  let $\varepsilon_X\colon \pi^\ast\pi_\ast X\to X$ be the counit. Applying $\pi_\ast$ and using
  the triangle identity together with the fact that $\eta_{\pi_\ast X}$ is an isomorphism, we
  obtain that $\pi_\ast(\varepsilon_X)$ is an isomorphism. Thus
  $\varepsilon_X\in W_{\arb{\mathcal C}}$, so the localized counit is also an isomorphism.
  Therefore~\eqref{eq:ho-equivalence} is an adjoint equivalence.
\end{proof}

\begin{example}\label{ex:R-linear-localization}
  Let $R$ be a ring, and set
  \[
    \sMod{R}:=(\Mod{R})^{\Delta_{\mathrm{epi}}}
    \qquad\text{and}\qquad
    \ArbMod{R}:=(\Mod{R})^{\OrdFor^{op}}.
  \]
  Suppose $W_{\sMod{R}}$ is a class of morphisms in $\sMod{R}$ defining a chosen homotopy theory of
  semisimplicial $R$-modules. Define a class of morphisms in $\ArbMod{R}$ by
  \[
    W_{\ArbMod{R}}
    :=
    \{\,f\in\Mor(\ArbMod{R})\mid \pi_\ast(f)\in W_{\sMod{R}}\,\}.
  \]
  Since $\Mod{R}$ is complete, Theorem~\ref{thm:localization-equivalence} applies and yields an
  adjoint equivalence
  \[
    \sMod{R}[W_{\sMod{R}}^{-1}]
    \simeq
    \ArbMod{R}[W_{\ArbMod{R}}^{-1}].
  \]
  Thus, once weak equivalences of arboreal $R$-modules are defined by detection through $\pi_\ast$,
  the homotopy category of arboreal $R$-modules is equivalent to that of semisimplicial $R$-modules.
\end{example}

\begin{remark}\label{rem:localization-consequence}
  Theorem~\ref{thm:localization-equivalence} applies whenever weak equivalences of arboreal objects
  are defined to be those morphisms detected by the right Kan extension $\pi_\ast$. In that sense,
  the homotopy theory of arboreal objects is determined by the semisimplicial shadow produced by
  $\pi_\ast$.
\end{remark}

\subsection{Automatic right-induced model structures}\label{sec:model-structure}

Assume now that $\mathcal C$ is bicomplete and that $\sC=\mathcal C^{\Delta_{\mathrm{epi}}}$ is
equipped with a cofibrantly generated model structure. Write
\[
  W_{\sC},\qquad \Fib_{\sC},\qquad \Cof_{\sC}
\]
for its weak equivalences, fibrations, and cofibrations. Since $\mathcal C$ is bicomplete,
precomposition along $\pi^{op}$ has both Kan extension adjoints, so that
$\pi_!\dashv \pi^\ast\dashv \pi_\ast$.

\begin{definition}\label{def:ordfor-weak-fib}
  A morphism $f\colon X\to Y$ in $\arb{\mathcal C}$ is called
  \begin{enumerate}[(i)]
  \item a \emph{weak equivalence} if $\pi_\ast f\in W_{\sC}$;
  \item a \emph{fibration} if $\pi_\ast f\in \Fib_{\sC}$.
  \end{enumerate}
  A \emph{cofibration} is a morphism having the left lifting property with respect to the acyclic
  fibrations.  Thus both weak equivalences and fibrations in $\arb{\mathcal C}$ are created by the
  right Kan extension $\pi_\ast$.
\end{definition}

\begin{remark}\label{rem:acyclicity-right-induced}
  The standard acyclicity condition for right-induced transfer along~\eqref{eq:adjoint-equivalence}
  is the inclusion
  \[
    {}^{\boxslash}\pi_\ast^{-1}(\Fib_{\sC})
    \subseteq
    \pi_\ast^{-1}(W_{\sC}).
  \]
  In our situation this condition is automatic, because the unit of $\pi^\ast\dashv \pi_\ast$ is an
  isomorphism by Lemma~\ref{lem:pi-star-fully-faithful}; see
  Proposition~\ref{prop:acyclicity-automatic}.
\end{remark}




\begin{proposition}\label{prop:acyclicity-automatic}
  Let $\mathcal M$ be a cofibrantly generated model category with generating acyclic cofibrations
  $J_{\mathcal M}$. Let $L\colon \mathcal M \rightleftarrows \mathcal N\colon U$ be an adjunction whose
  unit $\eta\colon \mathrm{Id}_{\mathcal M}\longrightarrow UL$ is an isomorphism. Then the standard
  acyclicity condition for right-induced transfer along $U$ is satisfied.
\end{proposition}

\begin{proof}
  By the standard transfer theorem for right-induced model structures, the acyclicity condition
  holds if and only if the left-transferred generating acyclic cofibrations $L(J_{\mathcal M})$ are
  weak equivalences in $\mathcal N$. In a right-induced model structure, a morphism $f$ in
  $\mathcal N$ is a weak equivalence if and only if $U(f)$ is a weak equivalence in $\mathcal
  M$. Therefore, we must verify that $U(L(j))$ is a weak equivalence in $\mathcal M$ for every
  $j \in J_{\mathcal M}$. For any such $j\colon A \to B$, the naturality of the unit $\eta$ yields
  a commutative square in $\mathcal M$ where the horizontal maps are $\eta_A$ and $\eta_B$, and the
  vertical maps are $j$ and $U(L(j))$. Because the unit is a natural isomorphism, the horizontal
  maps are isomorphisms. Since $j$ is a weak equivalence, the two-out-of-three property ensures
  that $U(L(j))$ is also a weak equivalence. Thus, $L(j)$ is a weak equivalence in $\mathcal N$,
  and the acyclicity condition holds automatically.
\end{proof}

\begin{theorem}\label{thm:right-induced-OrdFor}
  Let $\mathcal C$ be bicomplete, and assume that $\sC=\mathcal C^{\Delta_{\mathrm{epi}}}$ carries
  a cofibrantly generated model structure. Then $\arb{\mathcal C}$ admits a cofibrantly generated
  model structure right-induced from that on $\sC$ along $\pi_\ast$. Equivalently, the weak
  equivalences and fibrations in $\arb{\mathcal C}$ are exactly those of
  Definition~\ref{def:ordfor-weak-fib}.
\end{theorem}

\begin{proof}
  By Lemma~\ref{lem:pi-star-fully-faithful}, the unit of
  $\pi^\ast\colon \sC \rightleftarrows \arb{\mathcal C}:\pi_\ast$ is an isomorphism. Hence the
  acyclicity condition of Remark~\ref{rem:acyclicity-right-induced} holds by
  Proposition~\ref{prop:acyclicity-automatic}. The existence of the right-induced cofibrantly
  generated model structure therefore follows from the standard transfer theorem; see
  \cite[Theorem~2.2.1]{HKRS2017}.
\end{proof}

\begin{remark}\label{rem:generating-sets}
  If $I_{\sC}$ and $J_{\sC}$ are sets of generating cofibrations and generating acyclic
  cofibrations for the chosen model structure on $\sC$, then in the right-induced model structure
  on $\arb{\mathcal C}$ one may take $\pi^\ast(I_{\sC})$ and $\pi^\ast(J_{\sC})$ as sets of
  generating cofibrations and generating acyclic cofibrations.
\end{remark}

\begin{corollary}\label{cor:Reedy-specialization}
  Let $\mathcal C$ be a bicomplete cofibrantly generated model category. Since
  $\Delta_{\mathrm{inj}}$ is a direct Reedy category, the category
  $\mathcal C^{\Delta_{\mathrm{inj}}^{op}}$ of semisimplicial objects carries the Reedy model
  structure; see \cite[Chapter~15]{Hirschhorn2003} and \cite[Chapter~5]{Hovey1999}. Via the
  canonical isomorphism $\Delta_{\mathrm{inj}}^{op}\cong \Delta_{\mathrm{epi}}$, this yields a
  cofibrantly generated model structure on $\sC=\mathcal C^{\Delta_{\mathrm{epi}}}$. Hence
  $\arb{\mathcal C}$ carries the corresponding right-induced cofibrantly generated model structure
  along $\pi_\ast$.
\end{corollary}

\subsection{Quillen equivalence}\label{subsec:quillen-equivalence}

Assume that $\mathcal C$ is bicomplete, that $\sC=\mathcal C^{\Delta_{\mathrm{epi}}}$ carries a
cofibrantly generated model structure, and that $\arb{\mathcal C}=\mathcal C^{\OrdFor^{op}}$ is
equipped with the right-induced model structure of Theorem~\ref{thm:right-induced-OrdFor}. By
construction, weak equivalences and fibrations in $\arb{\mathcal C}$ are created by
$\pi_\ast\colon \arb{\mathcal C}\longrightarrow \sC$.

\begin{theorem}\label{thm:quillen-equivalence}
  Under these hypotheses, the adjunction
  $\pi^\ast\colon \sC \rightleftarrows \arb{\mathcal C}:\pi_\ast$ is a Quillen equivalence.
\end{theorem}

\begin{proof}
  Since the model structure on $\arb{\mathcal C}$ is right-induced along $\pi_\ast$, the right
  adjoint $\pi_\ast$ preserves fibrations and acyclic fibrations. Hence $\pi^\ast\dashv \pi_\ast$ is
  a Quillen adjunction.

  We apply \cite[Corollary~1.3.16(c)]{Hovey1999}. The functor $\pi_\ast$ reflects weak equivalences,
  hence in particular weak equivalences between fibrant objects, since weak equivalences in
  $\arb{\mathcal C}$ are defined via $\pi_\ast$.

  Now let $A$ be cofibrant in $\sC$, and let $\pi^\ast A \to R\pi^\ast A$ be a fibrant replacement
  in $\arb{\mathcal C}$. Applying $\pi_\ast$ gives a weak equivalence
  \[
    \pi_\ast\pi^\ast A \longrightarrow \pi_\ast R\pi^\ast A
  \]
  in $\sC$. By Lemma~\ref{lem:pi-star-fully-faithful}, the unit
  $\eta_A\colon A\to \pi_\ast\pi^\ast A$ is an isomorphism. Hence the composite
  \[
    A \xrightarrow{\ \cong\ } \pi_\ast\pi^\ast A \longrightarrow \pi_\ast R\pi^\ast A
  \]
  is a weak equivalence in $\sC$. Therefore the criterion of \cite[Corollary~1.3.16(c)]{Hovey1999}
  applies, and $\pi^\ast\dashv \pi_\ast$ is a Quillen equivalence.
\end{proof}

\begin{corollary}\label{cor:derived-equivalence}
  Under the hypotheses of Theorem~\ref{thm:quillen-equivalence}, the total left derived functor of
  $\pi^\ast$ and the total right derived functor of $\pi_\ast$ induce an equivalence of homotopy
  categories
  \[
    \mathbf L\pi^\ast\colon \Ho(\sC)\simeq \Ho(\arb{\mathcal C}):\mathbf R\pi_\ast.
  \]
\end{corollary}

\begin{proof}
  This is the defining consequence of a Quillen equivalence; see \cite[Definition~1.3.12 and
  Proposition~1.3.13]{Hovey1999}.
\end{proof}

\begin{example}\label{ex:R-linear-model}  
  Let $R$ be a ring, and write $\Ch_{\ge 0}(R)$ for the category of nonnegatively graded chain
  complexes of $R$-modules, equipped with the projective model structure.  Define a functor
  $N\colon \sMod{R}\longrightarrow \Ch_{\ge 0}(R)$ by
  \[
    N_n(X):=\bigcap_{i=0}^{n-1}\ker(d_i:X_n\to X_{n-1})
  \]
  where $\partial_n:=(-1)^n d_n\big|_{N_n(X)}$.  Its left adjoint
  $\Gamma\colon \Ch_{\ge 0}(R)\longrightarrow \sMod{R}$ is given by
  \[
    \Gamma(C)_n:=C_n,
    \qquad
    d_i=0\ (0\le i<n),
    \qquad
    d_n:=(-1)^n\partial_n.
  \]
  One checks directly that $\Gamma\dashv N$ and that the unit
  $\mathrm{Id}_{\Ch_{\ge 0}(R)}\longrightarrow N\Gamma$ is an isomorphism.  Therefore the standard
  acyclicity condition for right-induced transfer along $N$ is automatic. Hence $\sMod{R}$ admits
  the right-induced cofibrantly generated model structure from $\Ch_{\ge 0}(R)$ along $N$.
  Explicitly, a morphism $f$ in $\sMod{R}$ is
  \begin{enumerate}[(i)]
  \item a weak equivalence if $N(f)$ is a quasi-isomorphism;
  \item a fibration if $N(f)$ is degreewise surjective.
  \end{enumerate}

  Since $\pi^\ast\colon \sMod{R}\to \ArbMod{R}$ is fully faithful, the right-induced model structure on
  $\ArbMod{R}$ along
  \[
    \pi^\ast\colon \sMod{R}\rightleftarrows \ArbMod{R}:\pi_\ast
  \]
  also exists automatically, and this adjunction is a Quillen equivalence.
\end{example}

\subsection*{Acknowledgements}

The author acknowledges the use of large language models for copy-editing the manuscript.

\bibliographystyle{amsalpha}
\bibliography{references}

\end{document}